\documentclass[english,oneside,a4paper,11pt]{article}
\usepackage{amssymb}
\usepackage{latexsym}
\usepackage[english]{babel}
\usepackage{amsmath}
\usepackage{amsfonts}
\usepackage{amsbsy}
\usepackage[mathscr]{eucal}
\usepackage[latin1]{inputenc}
\usepackage{verbatim}
\usepackage{mathrsfs}
\usepackage{graphicx}
\usepackage{float}
\usepackage{bbm}
\usepackage{lscape}
\usepackage{mathdots}
\usepackage{enumitem}
\usepackage{empheq}
\usepackage[dvipsnames]{xcolor}
\usepackage{array}
\usepackage{a4wide}
\usepackage[dvips,colorlinks,bookmarks,breaklinks]{hyperref}  % PDF hyperlinks, with coloured links
\hypersetup{linkcolor=black,citecolor=black,filecolor=black,urlcolor=black}

\providecommand{\abs}[1]{\left\lvert#1 \right\rvert}
\providecommand{\Log}[1]{\textnormal{Log} #1}

\newtheorem{lemma}{Lemma}
\newtheorem{theorem}{Theorem}
\newtheorem{corollary}{Corollary}

\newcounter{counter}
\newcommand{\counter}{\stepcounter{counter}\thecounter}

\newenvironment{proof}
{\begin{trivlist}\item[\hskip%
\labelsep{{\it \noindent Proof.}}]}{\hfill $\square$
\end{trivlist}}

\newenvironment{proofTheorem1}
{\begin{trivlist}\item[\hskip%
\labelsep{{\it \noindent Proof of Theorem 1.}}]}{\hfill $\square$
\end{trivlist}}

\newenvironment{proofTheorem2}
{\begin{trivlist}\item[\hskip%
\labelsep{{\it \noindent Proof of Theorem 2.}}]}{\hfill $\square$
\end{trivlist}}

\newenvironment{proofCorollary1}
{\begin{trivlist}\item[\hskip%
\labelsep{{\it \noindent Proof of Corollary 1.}}]}{\hfill $\square$
\end{trivlist}}

\newenvironment{remark}
{\begin{trivlist}\item[\hskip%
\labelsep{{\it \noindent Remark \counter}}]}{\hfill
\end{trivlist}}

\numberwithin{equation}{section}

\allowdisplaybreaks

\begin{document}
\begin{center}
{\huge \textbf{Strong laws of large numbers for \\ arrays of row-wise extended negatively \\ dependent random variables}} \\
\vspace{1.5cm}
{\Large Jo\~{a}o Lita da Silva\footnote{\textit{E-mail address:} \texttt{jfls@fct.unl.pt}; \texttt{joao.lita@gmail.com}}} \\
\vspace{0.1cm}
\textit{Department of Mathematics and GeoBioTec \\ Faculty of Sciences and Technology \\
NOVA University of Lisbon \\ Quinta da Torre, 2829-516 Caparica,
Portugal}
\end{center}

\vspace{1.5cm}

\begin{abstract}
The main purpose of this paper is to obtain strong laws of large numbers for arrays or weighted sums of random variables under a scenario of dependence. Namely, for triangular arrays $\{X_{n,k}, \, 1 \leqslant k \leqslant n, \, n \geqslant 1 \}$ of row-wise extended negatively dependent random variables weakly mean dominated by a random variable $X \in \mathscr{L}_{1}$ and sequences $\{b_{n} \}$ of positive constants, conditions are given to ensure $\sum_{k=1}^{n} \left(X_{n,k} - \mathbb{E} \, X_{n,k} \right)/b_{n} \overset{\textnormal{a.s.}}{\longrightarrow} 0$. Our statements also allow us to improve recent results about complete convergence.
\end{abstract}

\bigskip


\bigskip


\bigskip

\section{Introduction}

In 1934, Harald Cram\'{e}r analyzed the almost sure convergence of the row sums of random arrays assuming the total independence of the random variables, thereby becoming a pioneer in the approach of this subject (see \cite{Cramer34}). Thenceforth, many authors have studied this challenging topic requiring always some independence on the arrays (see \cite{Baxter55},  \cite{Gut92}, \cite{Hu89}, or \cite{Teicher81} among others). A landmark paper in this context is \cite{Hu89}, where Hu, M\'{o}ricz and Taylor showed that for any triangular array $\{X_{n,k}, 1 \leqslant k \leqslant n, n \geqslant 1 \}$ of row-wise independent and zero-mean random variables uniformly bounded by a random variable $X$ satisfying $\mathbb{E} \abs{X}^{2p} < \infty$ for some $1 \leqslant p < 2$, $\sum_{k=1}^{n} X_{n,k}/n^{1/p}$ converges completely to zero (that is, for every $\varepsilon > 0$, $\sum_{n=1}^{\infty} \mathbb{P} \left\{\abs{\sum_{k=1}^{n} X_{n,k}/n^{1/p}} > \varepsilon \right\} < \infty$), and \emph{a fortiori}
\begin{equation*}
\frac{1}{n^{1/p}} \sum_{k=1}^{n} X_{n,k} \overset{\textnormal{a.s.}}{\longrightarrow} 0
\end{equation*}
by virtue of Borel-Cantelli lemma. Motivated by Hu, M\'{o}ricz and Taylor's result, Gut restated it for $0 < p < 2$ under a weaker distribution condition and at the expense of probability inequalities (see \cite{Gut92}, page $55$). Picking up this idea, we shall obtain general strong laws of large numbers for random triangular arrays having dependent structure, relaxing the independence assumption on the random variables. To achieve this goal, we shall employ a sharp exponential inequality of Bennett type to get the complete convergence towards zero of the referred random triangular arrays. Our approach will leads us not only to simpler and shorter proofs but also to improvements in some recent statements (e.g. Theorem 2.1 of \cite{Sung12}), which shows the tightness of our results.

We begin by retrieve a central definition along this paper announced by Gut in \cite{Gut92}. A random triangular array $\{X_{n,k}, 1 \leqslant k \leqslant n, n \geqslant 1 \}$ is said to be \emph{weakly mean dominated} by a random variable $X$ if, for some $C>0$,
\begin{equation*}
\frac{1}{n} \sum_{k=1}^{n} \mathbb{P} \left\{\abs{X_{n,k}} > t \right\} \leqslant C \, \mathbb{P} \left\{\abs{X} > t  \right\},
\end{equation*}
for all $t > 0$ and every $n \geqslant 1$. Let us point out that the above condition is weaker than the uniformly bounded condition assumed in \cite{Hu89} (see Example 2.1 of \cite{Gut92}). The following notion of dependence for triangular arrays of random variables was introduced in \cite{Lita16a} and will be essential throughout. A triangular array $\{X_{n,k}, \, 1 \leqslant k \leqslant n, \, n \geqslant 1 \}$ of random variables is said to be \emph{row-wise upper extended negatively dependent} (row-wise UEND) if for each $n \geqslant 1$, there exists a positive finite number $M_{n}$ such that
\begin{equation*}
\mathbb{P} \left(X_{n,1} > x_{1}, X_{n,2} > x_{2}, \ldots, X_{n,n} > x_{n} \right) \leqslant M_{n} \prod_{k=1}^{n} \mathbb{P} \left(X_{n,k} > x_{k} \right)
\end{equation*}
holds for all real numbers $x_{1}, \ldots, x_{n}$. A triangular array $\{X_{n,k}, \, 1 \leqslant k \leqslant n, \, n \geqslant 1 \}$ of random variables is said to be \emph{row-wise lower extended negatively dependent} (row-wise LEND) if for each $n \geqslant 1$, there exists a positive finite number $M_{n}$ such that
\begin{equation*}
\mathbb{P} \left(X_{n,1} \leqslant x_{1}, X_{n,2} \leqslant x_{2}, \ldots, X_{n,n} \leqslant x_{n} \right) \leqslant M_{n} \prod_{k=1}^{n} \mathbb{P} \left(X_{n,k} \leqslant x_{k} \right)
\end{equation*}
holds for all real numbers $x_{1}, \ldots, x_{n}$. A triangular array $\{X_{n,k}, \, 1 \leqslant k \leqslant n, \, n \geqslant 1 \}$ of random variables is said to be \emph{row-wise extended negatively dependent} (row-wise END) if it is both row-wise UEND and row-wise LEND. The sequence $\{M_{n}, \, n \geqslant 1 \}$ aforementioned is called a \emph{dominating sequence} of $\{X_{n,k}, \, 1 \leqslant k \leqslant n, \, n \geqslant 1 \}$ (see \cite{Lita16a}).

Lastly, we need to introduce also some relevant notations. Given a positive monotone sequence of constants $\{u_{n} \}$, a continuous monotone function $u( \, \cdot \,)$ on $[0,\infty[$ is called a \emph{monotone extension} of $\{u_{n} \}$ if $u(n) = u_{n}$ (see \cite{Chow97}, page 90); $u^{-1}$ should be interpreted as the generalized inverse of the extension $u$ when convenient. Associated to a probability space $(\Omega, \mathcal{F}, \mathbb{P})$, we shall consider the space $\mathscr{L}_{p}$ $(p > 0)$ of all measurable functions $X$ (necessarily random variables) for which $\mathbb{E} \abs{X}^{p} < \infty$. The letter $C$ will denote a positive constant, which is not necessarily the same one in each appearance; symbols $C(\varepsilon)$ or $C(\delta)$ have the same meaning with the additional information that they depend on $\varepsilon$ or $\delta$, respectively. The notation $\lfloor x \rfloor$ will be used to indicate the largest integer not greater than $x$ and $\Log \, x$ will denote $\log \max\{x,e \}$.

\section{Main results}

Our first major result in this sequel is a general strong law of large numbers for triangular arrays of random variables having dependent structure and allow us weaken or strengthen the assumptions on the random variables through integrability conditions.

\begin{theorem}\label{thr:1}
Let $\{X_{n,k}, \, 1 \leqslant k \leqslant n, \, n \geqslant 1 \}$ be a triangular array of row-wise END random variables with dominating sequence $\{M_{n}, \, n \geqslant 1 \}$ weakly mean dominated by a random variable $X \in \mathscr{L}_{1}$, $\{a_{n} \}$ a positive increasing sequence of constants with increasing extension $a(\, \cdot \,)$ and $\{b_{n} \}$, $\{s_{n} \}$ positive nondecreasing sequences of constants with nondecreasing extensions $b(\, \cdot \,)$, $s ( \, \cdot \, )$ respectively. If
\begin{itemize}[align=left]
\item[\textnormal{(a)}] ${\displaystyle \sum_{k=1}^{n} \left[\mathbb{E} \left( X_{n,k}^{2}I_{\left\{\lvert X_{n,k} \rvert \leqslant a_{n} \right\}} \right) + a_{n}^{2} \mathbb{P} \left\{\lvert X_{n,k} \rvert > a_{n} \right\} \right] \leqslant s_{n}}$,

\item[\textnormal{(b)}] ${\displaystyle \frac{s_{n}}{a_{n} b_{n}} = o(1)}$ as $n \rightarrow \infty$,

\item[\textnormal{(c)}] ${\displaystyle \frac{s_{n}}{a_{n}^{2} \, \Log \, n} = o(1)}$ as $n \rightarrow \infty$,

\item[\textnormal{(d)}] ${\displaystyle \liminf_{n \rightarrow \infty} \frac{b_{n} \, \Log \left(a_{n} b_{n}/s_{n} \right)}{a_{n} \, \Log \left(n^{\delta} \right)}  > 1}$ for all $\delta > 0$,

\item[\textnormal{(e)}] ${\displaystyle \int_{0}^{\infty} \frac{1}{\Log \, u} \int_{\lfloor u \rfloor}^{\infty} \Log \left[\frac{a(t) b(t)}{s(t)} \right] \mathbb{P} \left\{a^{-1} \left(\abs{X} \right) > t \right\} \mathrm{d}t \mathrm{d}u < \infty}$,

\item[\textnormal{(f)}] ${\displaystyle \int_{0}^{\infty} \mathbb{P} \left\{\abs{X} > t \right\} \int_{0}^{\lfloor a^{-1}(t) \rfloor} \frac{u}{b(u)} \mathrm{d}u \, \mathrm{d}t < \infty}$,

\item[\textnormal{(g)}] $M_{n} = O\left(n^{\alpha} \right)$ as $n \rightarrow \infty$ for some $\alpha > 0$,
\end{itemize}
then
\begin{equation*}
\sum_{n=1}^{\infty} \mathbb{P} \left\{\abs{\frac{1}{b_{n}} \sum_{k=1}^{n} \left(X_{n,k} - \mathbb{E} \, X_{n,k} \right)} > \varepsilon \right\} < \infty \quad \text{for all} \; \; \varepsilon > 0,
\end{equation*}
and ${\displaystyle \frac{1}{b_{n}} \sum_{k=1}^{n} \left(X_{n,k} - \mathbb{E} \, X_{n,k} \right) \overset{\textnormal{a.s.}}{\longrightarrow} 0}$.
\end{theorem}

\begin{remark}
Alternatively, condition (a) of Theorem~\ref{thr:1} can be written as
\begin{equation*}
\sum_{k=1}^{n} \int_{0}^{a_{n}} u \mathbb{P} \left\{\lvert X_{n,k} \rvert > u \right\} \mathrm{d}u \leqslant \frac{s_{n}}{2}.
\end{equation*}
\end{remark}

The next corollary is a strong law of large numbers for (weighted) arrays of row-wise END random variables that preserves both norming constants and moment condition assumed in \cite{Hu89}. Particularly, it broadens Theorem 2.1 of \cite{Gut92} to the herein referred dependent structures of random variables (by taking $c_{n,k} = 1$ for each $n,k$).

\begin{corollary}\label{cor:1}
Let $\{X_{n,k}, \, 1 \leqslant k \leqslant n, \, n \geqslant 1 \}$ be a triangular array of row-wise END random variables with dominating sequence $\{M_{n}, \, n \geqslant 1 \}$ weakly mean dominated by a (non null) random variable $X$ such that $\mathbb{E} \abs{X}^{2p} < \infty$ for some $0 <  p < 2$, and $M_{n} = O\left(n^{\alpha} \right)$ as $n \rightarrow \infty$ for some $\alpha > 0$. If $\left\{c_{n,k}, \, 1 \leqslant k \leqslant n, n \geqslant 1 \right\}$ is an array of constants such that
\begin{equation*}
\max_{1 \leqslant k \leqslant n} \abs{c_{n,k}} = O(1), \quad n \rightarrow \infty
\end{equation*}
then
\begin{equation*}
\sum_{n=1}^{\infty} \mathbb{P} \left\{\abs{\frac{1}{n^{1/p}} \sum_{k=1}^{n} c_{n,k} \left(X_{n,k} - \mathbb{E} \, X_{n,k} \right)} > \varepsilon \right\} < \infty \quad \text{for all} \; \; \varepsilon > 0,
\end{equation*}
and ${\displaystyle \frac{1}{n^{1/p}} \sum_{k=1}^{n} c_{n,k} \left(X_{n,k} - \mathbb{E} \, X_{n,k} \right) \overset{\textnormal{a.s.}}{\longrightarrow} 0}$.
\end{corollary}

\begin{remark}
The previous statement extends Theorem 2.1 of \cite{Sung12} not only allowing $p < 1$ and enlarging the class of random triangular arrays (recall that arrays of row-wise negatively dependent random variables are arrays of row-wise END random variables with $M_{n} = 1$ for all $n$) but also discarding its condition (2.4). In fact, supposing $\{X_{n,k}, \, 1 \leqslant k \leqslant n, \, n \geqslant 1 \}$ and $\{a_{n,k}, \, 1 \leqslant k \leqslant n, \, n \geqslant 1 \}$ as in Theorem 2.1 of \cite{Sung12}, and $c_{n,k} = n^{1/p}  a_{n,k}$ in Corollary~\ref{cor:1} we get that $\sum_{k=1}^{n} a_{n,k} X_{n,k}$ converges completely to zero provided only $\max_{1 \leqslant k \leqslant n} \abs{a_{n,k}} = O \left(n^{-1/p} \right)$, $n \rightarrow \infty$. Furthermore, Corollary~\ref{cor:1} still improves assumption (4.11) and the moment condition presented in Corollary 4.4 of \cite{Shen16}.
\end{remark}

Our last result extends Theorem 1 of \cite{Lita15} to widely orthant dependent sequences of random variables with dominating sequence $\{M_{n}, \, n \geqslant 1 \}$, that is, to random sequences $\{X_{n}, \, n \geqslant 1 \}$ such that, for each $n \geqslant 1$, there exists some finite positive number $M_{n}$ satisfying
\begin{equation}\label{eq:2.1}
\mathbb{P} \left(\bigcap_{k=1}^{n} \left\{X_{k} > x_{k} \right\} \right) \leqslant M_{n} \prod_{k=1}^{n} \mathbb{P} \left\{X_{k} > x_{k} \right\}
\end{equation}
and
\begin{equation}\label{eq:2.2}
\mathbb{P} \left(\bigcap_{k=1}^{n} \left\{X_{k} \leqslant x_{k} \right\} \right) \leqslant M_{n} \prod_{k=1}^{n} \mathbb{P} \left\{X_{k} \leqslant x_{k} \right\}
\end{equation}
for all real numbers $x_{1}, \ldots,x_{n}$ (see \cite{Chen13}, page $116$). Note that in \eqref{eq:2.1} and \eqref{eq:2.2} we are taking $M_{n} := \max\{g_{U}(n), g_{L}(n) \}$ with $g_{U}(n)$ and $g_{L}(n)$ as in Definition 1.1 of \cite{Chen13}.

\begin{theorem}\label{thr:2}
If $\{X_{n}, \, n \geqslant 1 \}$ is a sequence of widely orthant dependent random variables with dominating sequence $\{M_{n}, \, n \geqslant 1 \}$ satisfying $M_{n} = O\left(n^{\alpha} \right)$, $n \rightarrow \infty$  for some $\alpha > 0$, stochastically dominated by a random variable $X \in \mathscr{L}_{p}$ for some $1 < p < 2$, and $\left\{c_{n,k}, \, 1 \leqslant k \leqslant n, n \geqslant 1 \right\}$ is an array of constants such that
\begin{equation}\label{eq:2.3}
\max_{1 \leqslant k \leqslant n} \abs{c_{n,k}} = O(1), \quad n \rightarrow \infty
\end{equation}
then
\begin{equation*}
\frac{1}{n^{1/p} \, \Log^{1 - 1/p} n} \sum_{k=1}^{n} c_{n,k} \left(X_{k} - \mathbb{E} \, X_{k} \right) \overset{\textnormal{a.s.}}{\longrightarrow} 0.
\end{equation*}
\end{theorem}

\section{Lemmas and proofs}

We begin this section by presenting a Bennett inequality type (see \cite{Bennett62}) for triangular arrays of row-wise UEND random variables with dominating sequence $\{M_{n}, \, n \geqslant 1 \}$ which plays a central role in subsequent proofs.

\begin{lemma}\label{lem:1}
Let $\{X_{n,k}, \, 1 \leqslant k \leqslant n, \, n \geqslant 1 \}$ be a triangular array of zero-mean row-wise UEND random variables with dominating sequence $\{M_{n}, \, n \geqslant 1 \}$ and $\{a_{n} \}$, $\{s_{n} \}$ sequences of positive constants. If $X_{n,k} \leqslant a_{n}$ a.s. for every $1 \leqslant k \leqslant n$, $n \geqslant 1$ and $\sum_{k=1}^{n} \mathbb{E} \, X_{n,k}^{2} \leqslant s_{n}$ then
\begin{equation*}
\forall \varepsilon > 0, \quad \mathbb{P} \left\{\sum_{k=1}^{n} X_{n,k} > \varepsilon \right\} \leqslant M_{n} \exp \left[\frac{\varepsilon}{a_{n}} - \left(\frac{\varepsilon}{a_{n}} + \frac{s_{n}}{a_{n}^{2}} \right) \log \left(1 + \frac{\varepsilon a_{n}}{s_{n}} \right) \right].
\end{equation*}
\end{lemma}

\begin{proof}
Consider the function $g\colon \mathbb{R} \longrightarrow \mathbb{R}$ defined by $g(x) = x^{-2} \left(e^{x} - 1 - x \right)$, $x \neq 0$ and $g(0) = 1/2$. Since $g$ is nonnegative, increasing and convex on $\mathbb{R}$ (see \cite{Teicher79}, page $295$), we have
\begin{equation}\label{eq:3.1}
\mathbb{E} \exp \left(t_{n} X_{n,k} \right) \leqslant 1 + t_{n}^{2} \, g \left(t_{n} a_{n} \right) \mathbb{E} X_{n,k}^{2} \leqslant \exp \left[\frac{\exp\left(t_{n} a_{n} \right) - 1 - t_{n} a_{n}}{a_{n}^{2}} \, \mathbb{E} X_{n,k}^{2} \right]
\end{equation}
for any sequence $\{t_{n} \}$ of positive constants. Since $\{X_{n,k}, \, 1 \leqslant k \leqslant n, \, n \geqslant 1 \}$ is row-wise UEND with dominating sequence $\{M_{n}, \, n \geqslant 1 \}$ we obtain
\begin{equation*}
\mathbb{E} \exp \left(t_{n} \sum_{k=1}^{n} X_{n,k} \right) \leqslant M_{n} \prod_{k=1}^{n} \mathbb{E} \exp \left(t_{n} X_{n,k} \right)
\end{equation*}
via Lemma 1 of \cite{Lita16b} with $f_{n,k}(x) = e^{t_{n}x}$ $(t_{n} > 0)$, and from \eqref{eq:3.1} we get
\begin{align*}
\mathbb{E} \exp \left(t_{n} \sum_{k=1}^{n} X_{n,k} \right) &\leqslant M_{n} \exp \left[\frac{\exp\left(t_{n} a_{n} \right) - 1 - t_{n} a_{n}}{a_{n}^{2}} \sum_{k=1}^{n} \mathbb{E} \, X_{n,k}^{2} \right] \\
&\leqslant M_{n} \exp \left[\frac{\exp\left(t_{n} a_{n} \right) - 1 - t_{n} a_{n}}{a_{n}^{2}} s_{n} \right].
\end{align*}
Fixing $\varepsilon > 0$ arbitrarily we have
\begin{align*}
\mathbb{P} \left\{\sum_{k=1}^{n} X_{n,k} > \varepsilon \right\} &\leqslant \exp \left(-\varepsilon t_{n} \right) \mathbb{E} \exp \left(t_{n} \sum_{k=1}^{n} X_{n,k} \right) \\
&\leqslant M_{n} \exp \left[-\varepsilon t_{n} + \frac{\exp\left(t_{n} a_{n} \right) - 1 - t_{n} a_{n}}{a_{n}^{2}} s_{n} \right]
\end{align*}
according to Chebyshev inequality. The right-hand side of the above inequality is minimized when $t_{n} = \log \left(1 + \dfrac{\varepsilon a_{n}}{s_{n}} \right)^{\frac{1}{a_{n}}}$ which yields
\begin{equation*}
\mathbb{P} \left\{\sum_{k=1}^{n} X_{n,k} > \varepsilon \right\} \leqslant M_{n} \exp \left[\frac{\varepsilon}{a_{n}} - \left(\frac{\varepsilon}{a_{n}} + \frac{s_{n}}{a_{n}^{2}} \right) \log \left(1 + \frac{\varepsilon a_{n}}{s_{n}} \right) \right].
\end{equation*}
The proof is complete.
\end{proof}

\medskip

For the sake of a comparison of Bennet's inequality and Lemma 2 in \cite{Lita16a} (i.e. Bernstein's inequality), suppose that $\lvert X_{n,k} \rvert \leqslant a_{n}$ a.s. for any $1 \leqslant k \leqslant n$, $n \geqslant 1$ and $s_{n} := \sum_{k=1}^{n} \mathbb{E} \, X_{n,k}^{2}$. Hence, all assumptions of the aforementioned Lemma~\ref{lem:1} are satisfied. Moreover, the conditions in Lemma 2 of \cite{Lita16a} are also verified, and
\begin{equation*}
-\frac{\varepsilon^{2}}{2(\varepsilon a_{n} + s_{n})} > \frac{\varepsilon}{a_{n}} - \left(\frac{\varepsilon}{a_{n}} + \frac{s_{n}}{a_{n}^{2}} \right) \log \left(1 + \frac{\varepsilon a_{n}}{s_{n}} \right)
\end{equation*}
for all $\varepsilon \geqslant 5 s_{n}/a_{n}$; indeed, it is straightforward to see that the function
$x \mapsto 2 + x/(1 + x) - 2(1 - 1/x) \log(1 + x)$ is negative and non-increasing for all $x \geqslant 5$ being also asymptotically equivalent to $-2 \log x$ as $x \rightarrow \infty$. Therefore, it follows that for large values of $\varepsilon$ Bennett's bound is sharper than Bernstein's bound.

The statement below is a Fuk-Nagaev inequality type (see \cite{Fuk71}) announced for arrays of row-wise END random variables. The proof follows the same steps of the original one in \cite{Fuk71}.

\begin{lemma}\label{lem:2}
Let $0 < p \leqslant 1$. If $\{X_{n,k}, \, 1 \leqslant k \leqslant n, \, n \geqslant 1 \}$ is a triangular array of row-wise END random variables with dominating sequence $\{M_{n}, \, n \geqslant 1 \}$ such that $\mathbb{E} \abs{X_{n,k}}^{p} < \infty$, for all $1 \leqslant k \leqslant n$, $n \geqslant 1$  then for all $\varepsilon, \lambda > 0$,
\begin{equation*}
\mathbb{P} \left\{\abs{\sum_{k=1}^{n} X_{n,k}} > \varepsilon \right\} \leqslant \sum_{k=1}^{n} \mathbb{P} \left\{\abs{X_{n,k}} > \frac{\varepsilon}{\lambda} \right\} +  2 M_{n} e^{\lambda} \left(1 + \frac{\varepsilon^{p}}{\lambda^{p-1} \sum_{k=1}^{n} \mathbb{E} \abs{X_{n,k}}^{p}}  \right)^{-\lambda}.
\end{equation*}
\end{lemma}

\begin{proof}
Let $\{ \delta_{n} \}$ be a sequence of positive constants and consider the random variables $T_{n,k} := \min \left(X_{n,k}, \delta_{n} \right)$, $1 \leqslant k \leqslant n$, $n \geqslant 1$. Hence,
\begin{equation*}
\left\{\omega\colon \sum_{k=1}^{n} X_{n,k} > \varepsilon \right\} \subset \left\{\omega\colon \sum_{k=1}^{n} T_{n,k} \neq \sum_{k=1}^{n} X_{n,k} \right\} \cup \left\{\omega\colon \sum_{k=1}^{n} T_{n,k} > \varepsilon \right\}
\end{equation*}
and for all $t_{n} > 0$,
\begin{equation}\label{eq:3.2}
\begin{split}
\mathbb{P} \left\{\sum_{k=1}^{n} X_{n,k} > \varepsilon \right\} &\leqslant \mathbb{P} \left\{\sum_{k=1}^{n} T_{n,k} \neq \sum_{k=1}^{n} X_{n,k} \right\} + \mathbb{P} \left\{\sum_{k=1}^{n} T_{n,k} > \varepsilon \right\} \\
&\leqslant \sum_{k=1}^{n} \mathbb{P} \left\{X_{n,k} > \delta_{n} \right\} + \exp \left(-\varepsilon t_{n} \right) \mathbb{E} \exp \left(t_{n} \sum_{k=1}^{n} T_{n,k} \right) \\
&\leqslant \sum_{k=1}^{n} \mathbb{P} \left\{X_{n,k} > \delta_{n} \right\} + M_{n} \exp \left(-\varepsilon t_{n} \right) \prod_{k=1}^{n} \mathbb{E} \exp \left(t_{n} T_{n,k} \right)
\end{split}
\end{equation}
provided that $\left\{T_{n,k}, \, 1 \leqslant k \leqslant n, \, n \geqslant 1 \right\}$ is row-wise END (see Lemma 1 of \cite{Lita16b}). Fixing $0 < p \leqslant 1$, we obtain
\begin{align*}
\mathbb{E} \exp \left(t_{n} T_{n,k} \right) &= \int_{-\infty}^{\delta_{n}} \left(e^{t_{n}u} - 1 \right) \mathrm{d}\mathbb{P} \left\{T_{n,k} \leqslant u \right\} + \int_{\delta_{n}}^{\infty} \left(e^{t_{n} \delta_{n}} - 1 \right) \mathrm{d}\mathbb{P} \left\{T_{n,k} \leqslant u \right\} + 1 \\
&\leqslant \int_{0}^{\delta_{n}} \left(e^{t_{n}u} - 1 \right) \mathrm{d}\mathbb{P} \left\{T_{n,k} \leqslant u \right\} + \int_{\delta_{n}}^{\infty} \left(e^{t_{n} \delta_{n}} - 1 \right) \mathrm{d}\mathbb{P} \left\{T_{n,k} \leqslant u \right\} + 1 \\
&\leqslant \frac{e^{t_{n} \delta_{n}} - 1}{\delta_{n}^{p}} \int_{0}^{\delta_{n}} u^{p} \mathrm{d}\mathbb{P} \left\{T_{n,k} \leqslant u \right\} + \frac{e^{t_{n} \delta_{n}} - 1}{\delta_{n}^{p}} \int_{\delta_{n}}^{\infty} u^{p} \mathrm{d}\mathbb{P} \left\{T_{n,k} \leqslant u \right\} + 1 \\
&\leqslant 1 + \frac{e^{t_{n} \delta_{n}} - 1}{\delta_{n}^{p}} \mathbb{E} \abs{T_{n,k}}^{p} \\
&\leqslant \exp \left(\frac{e^{t_{n} \delta_{n}} - 1}{\delta_{n}^{p}} \mathbb{E} \abs{X_{n,k}}^{p} \right)
\end{align*}
since, for each $n \geqslant 1$, the function $u \mapsto \left(e^{t_{n} u} - 1 \right)/u^{p}$ is nondecreasing on $(0,\infty)$. From the latter inequality and \eqref{eq:3.2}, we get
\begin{equation}\label{eq:3.3}
\mathbb{P} \left\{\sum_{k=1}^{n} X_{n,k} > \varepsilon \right\}\leqslant \sum_{k=1}^{n} \mathbb{P} \left\{X_{n,k} > \delta_{n} \right\} + M_{n} \exp \left[-\varepsilon t_{n} + \frac{e^{t_{n} \delta_{n}} - 1}{\delta_{n}^{p}} \sum_{k=1}^{n} \mathbb{E} \abs{X_{n,k}}^{p} \right].
\end{equation}
Setting $s_{n,p} := \sum_{k=1}^{n} \mathbb{E} \abs{X_{n,k}}^{p}$ and taking $t_{n} = \log \left(1 + \varepsilon \delta_{n}^{p - 1}/\sum_{k=1}^{n} \mathbb{E} \abs{X_{n,k}}^{p} \right)^{1/\delta_{n}}$ in \eqref{eq:3.3}, it follows
\begin{equation*}
\mathbb{P} \left\{\sum_{k=1}^{n} X_{n,k} > \varepsilon \right\} \leqslant \sum_{k=1}^{n} \mathbb{P} \left\{X_{n,k} > \delta_{n} \right\} + M_{n} \exp \left[\frac{\varepsilon}{\delta_{n}} - \frac{\varepsilon}{\delta_{n}} \log \left(1 + \frac{\varepsilon \delta_{n}^{p - 1}}{s_{n,p}} \right) \right].
\end{equation*}
Replacing $X_{n,k}$ by $-X_{n,k}$ and noting that, by Lemma 1 of \cite{Lita16b}, $\left\{-X_{n,k}, \, 1 \leqslant k \leqslant n, \, n \geqslant 1 \right\}$ is still an array of zero-mean row-wise END random variables with dominating sequence $\{M_{n}, \, n \geqslant 1 \}$ satisfying $\mathbb{E} \abs{X_{n,k}}^{p} < \infty$, for all $1 \leqslant k \leqslant n$, $n \geqslant 1$, we have
\begin{equation*}
\mathbb{P} \left\{- \sum_{k=1}^{n} X_{n,k} > \varepsilon \right\} \leqslant \sum_{k=1}^{n} \mathbb{P} \left\{- X_{n,k} > \delta_{n} \right\} + M_{n} \exp \left[\frac{\varepsilon}{\delta_{n}} - \frac{\varepsilon}{\delta_{n}} \log \left(1 + \frac{\varepsilon \delta_{n}^{p - 1}}{s_{n,p}} \right) \right]
\end{equation*}
and
\begin{equation}\label{eq:3.4}
\mathbb{P} \left\{\abs{\sum_{k=1}^{n} X_{n,k}} > \varepsilon \right\} \leqslant \sum_{k=1}^{n} \mathbb{P} \left\{\abs{X_{n,k}} > \delta_{n} \right\} + 2 M_{n} \exp \left[\frac{\varepsilon}{\delta_{n}} - \frac{\varepsilon}{\delta_{n}} \log \left(1 + \frac{\varepsilon \delta_{n}^{p - 1}}{s_{n,p}} \right) \right].
\end{equation}
Considering $\delta_{n} = \varepsilon/\lambda$ $(\lambda > 0)$ in \eqref{eq:3.4}, yields
\begin{equation*}
\mathbb{P} \left\{\abs{\sum_{k=1}^{n} X_{n,k}} > \varepsilon \right\} \leqslant \sum_{k=1}^{n} \mathbb{P} \left\{\abs{X_{n,k}} > \frac{\varepsilon}{\lambda} \right\} + 2 M_{n} e^{\lambda} \left(1 + \frac{\varepsilon^{p}}{\lambda^{p - 1} s_{n,p}} \right)^{-\lambda}
\end{equation*}
finishing the proof.
\end{proof}

\begin{lemma}\label{lem:3}
Let $\{X_{n,k}, \, 1 \leqslant k \leqslant n, \, n \geqslant 1 \}$ be a triangular array of zero-mean row-wise END random variables with dominating sequence $\{M_{n}, \, n \geqslant 1 \}$ and $\{a_{n}\}$, $\{b_{n} \}$, $\{s_{n} \}$ sequences of positive constants. If
\begin{itemize}[align=left]
\item[\textnormal{(i)}] $\abs{X_{n,k}} \leqslant a_{n}$ a.s. for every $1 \leqslant k \leqslant n$, $n \geqslant 1$,

\item[\textnormal{(ii)}] ${\displaystyle \sum_{k=1}^{n} \mathbb{E} \, X_{n,k}^{2} \leqslant s_{n}}$,

\item[\textnormal{(iii)}] ${\displaystyle \frac{s_{n}}{a_{n} b_{n}} = o(1)}$ as $n \rightarrow \infty$,

\item[\textnormal{(iv)}] ${\displaystyle \frac{s_{n}}{a_{n}^{2} \, \Log \, n} = o(1)}$ as $n \rightarrow \infty$,

\item[\textnormal{(v)}] ${\displaystyle \liminf_{n \rightarrow \infty} \frac{b_{n} \, \Log \left(a_{n} b_{n}/s_{n} \right)}{a_{n} \, \Log \left(n^{\delta} \right)}  > 1}$ for all $\delta > 0$,

\item[\textnormal{(vi)}] $M_{n} = O\left(n^{\alpha} \right)$ as $n \rightarrow \infty$ for some $\alpha > 0$,
\end{itemize}
then
\begin{equation*}
\sum_{n=1}^{\infty} \mathbb{P} \left\{\abs{\frac{1}{b_{n}} \sum_{k=1}^{n} X_{n,k}} > \varepsilon \right\} < \infty \quad \text{for all} \; \; \varepsilon > 0.
\end{equation*}
\end{lemma}

\begin{proof}
Fix arbitrarily $\varepsilon > 0$. We have
\begin{align*}
\mathbb{P} & \left\{\frac{1}{b_{n}} \sum_{k=1}^{n} X_{n,k} > \varepsilon \right\} \leqslant \\
& \leqslant M_{n} \exp \left[\frac{\varepsilon b_{n}}{a_{n}} - \left(\frac{\varepsilon b_{n}}{a_{n}} + \frac{s_{n}}{a_{n}^{2}} \right) \log \left(1 + \frac{\varepsilon a_{n} b_{n}}{s_{n}} \right) \right] \\
& = \exp \left\{\left[\frac{\log M_{n}}{\log n} - \left(\frac{\varepsilon b_{n}}{a_{n} \log n} + \frac{s_{n}}{a_{n}^{2} \log n} \right) \left(\log \left(1 + \frac{\varepsilon a_{n} b_{n}}{s_{n}} \right) - 1 \right) - \frac{s_{n}}{a_{n}^{2} \log n} \right] \log n \right\}
\end{align*}
according to Lemma~\ref{lem:1}. From conditions (iii) and (iv) we obtain
\begin{equation*}
\left(\frac{\varepsilon b_{n}}{a_{n} \log n} + \frac{s_{n}}{a_{n}^{2} \log n} \right) \left[\log \left(1 + \frac{\varepsilon a_{n} b_{n}}{s_{n}} \right) - 1 \right] \geqslant \frac{b_{n} \, \Log \left(a_{n} b_{n}/s_{n} \right)}{C(\varepsilon) a_{n} \, \Log \, n}
\end{equation*}
for all sufficiently large $n$ and some $C(\varepsilon)  > 0$. Thus, conditions (iv), (v) and (vi) yield
\begin{align*}
\limsup_{n \rightarrow \infty} & \left\{\frac{\log M_{n}}{\log n} - \left(\frac{\varepsilon b_{n}}{a_{n} \log n} + \frac{s_{n}}{a_{n}^{2} \log n} \right) \left[\log \left(1 + \frac{\varepsilon a_{n} b_{n}}{s_{n}} \right) - 1 \right] - \frac{s_{n}}{a_{n}^{2} \log n} \right\} \leqslant \\
&\qquad \leqslant \alpha - \liminf_{n \rightarrow \infty} \frac{b_{n} \, \Log \left(a_{n} b_{n}/s_{n} \right)}{C(\varepsilon) a_{n} \, \Log \, n} < - 1,
\end{align*}
for some $\alpha > 0$ (fixed) since ${\displaystyle \liminf_{n \rightarrow \infty} \frac{b_{n} \, \Log \left(a_{n} b_{n}/s_{n} \right)}{(1 + \alpha)C(\varepsilon) a_{n} \, \Log \, n} > 1}$. Thereby,
\begin{equation}\label{eq:3.5}
\sum_{n=1}^{\infty} \mathbb{P} \left\{\frac{1}{b_{n}} \sum_{k=1}^{n} X_{n,k} > \varepsilon \right\} < \infty \quad \text{for all} \; \; \varepsilon > 0.
\end{equation}
According to Lemma 1 of \cite{Lita16b}, $\{-X_{n,k}, \, 1 \leqslant k \leqslant n, \, n \geqslant 1\}$ is still row-wise END with dominating sequence $\{M_{n}, \, n \geqslant 1 \}$. Hence, performing similar computations for the triangular array $\{-X_{n,k}, \, 1 \leqslant k \leqslant n, \, n \geqslant 1\}$, we get
\begin{equation}\label{eq:3.6}
\sum_{n=1}^{\infty} \mathbb{P} \left\{- \frac{1}{b_{n}} \sum_{k=1}^{n} X_{n,k} > \varepsilon \right\} < \infty \quad \text{for all} \; \; \varepsilon > 0.
\end{equation}
The result follows by \eqref{eq:3.5} and \eqref{eq:3.6}.
\end{proof}

\begin{proofTheorem1}
Setting
\begin{gather*}
X_{n,k}' = X_{n,k} I_{\left\{\abs{X_{n,k}} \leqslant a_{n} \right\}} + a_{n} I_{\left\{X_{n,k} > a_{n} \right\}} - a_{n} I_{\left\{X_{n,k} < - a_{n} \right\}}, \\
X_{n,k}'' = X_{n,k} I_{\left\{\abs{X_{n,k}} > a_{n} \right\}} + a_{n} I_{\left\{X_{n,k} < - a_{n} \right\}} - a_{n} I_{\left\{X_{n,k} > a_{n} \right\}}
\end{gather*}
we have $X_{n,k}' + X_{n,k}'' = X_{n,k}$. From Lemma 1 of \cite{Lita16b}, the triangular array $\{X_{n,k}', 1 \leqslant k \leqslant n, n \geqslant 1 \}$ is row-wise END with dominating sequence $\{M_{n}, \, n \geqslant 1 \}$ since the function $T_{\ell}(t) = \max(\min(t,\ell),-\ell)$, which describes the truncation at level $\ell$, is nondecreasing. Further, $\{X_{n,k}' - \mathbb{E} \, X_{n,k}', 1 \leqslant k \leqslant n, n \geqslant 1 \}$ is also row-wise END with dominating sequence $\{M_{n}, \, n \geqslant 1 \}$ and
\begin{equation*}
\big\vert X_{n,k}' - \mathbb{E} \, X_{n,k}' \big\vert \leqslant 2 a_{n}.
\end{equation*}
Since $\sum_{k=1}^{n} \mathbb{E} \big(X_{n,k}' - \mathbb{E} \, X_{n,k}' \big)^{2} \leqslant 2 s_{n}$, Lemma~\ref{lem:3} guarantees
\begin{equation}\label{eq:3.7}
\sum_{n=1}^{\infty} \mathbb{P} \left\{\abs{\frac{1}{b_{n}} \sum_{k=1}^{n} \left(X_{n,k}' - \mathbb{E} \, X_{n,k}' \right)} > \varepsilon \right\} < \infty \quad \text{for all} \; \; \varepsilon > 0.
\end{equation}
Now, we shall demonstrate that
\begin{equation}\label{eq:3.8}
\sum_{n=1}^{\infty} \mathbb{P} \left\{\abs{\frac{1}{b_{n}} \sum_{k=1}^{n} \left(X_{n,k}'' - \mathbb{E} \, X_{n,k}'' \right)} > \varepsilon \right\} < \infty \quad \text{for all} \; \; \varepsilon > 0.
\end{equation}
We have $\vert X_{n,k}'' \vert \leqslant \vert X_{n,k} \vert I_{\left\{\abs{X_{n,k}} > a_{n} \right\}}$ and
\begin{align*}
\frac{1}{n} \sum_{k=1}^{n} \mathbb{E} \vert X_{n,k} \vert I_{\left\{\abs{X_{n,k}} > a_{n} \right\}} \leqslant C \, \mathbb{E} \abs{X} I_{\left\{\abs{X} > a_{n} \right\}}
\end{align*}
since $\{X_{n,k}, 1 \leqslant k \leqslant n, n \geqslant 1 \}$ is weakly mean dominated by $X$. Thus
\begin{align*}
\mathbb{P} \left\{\abs{\frac{1}{b_{n}} \sum_{k=1}^{n} \left(X_{n,k}'' - \mathbb{E} \, X_{n,k}'' \right)} > \varepsilon \right\} &\leqslant \mathbb{P} \left\{\frac{1}{b_{n}} \sum_{k=1}^{n} \vert X_{n,k}'' - \mathbb{E} \, X_{n,k}'' \vert > \varepsilon \right\} \\
&\leqslant \frac{2}{\varepsilon b_{n}} \sum_{k=1}^{n} \mathbb{E} \vert X_{n,k} \vert I_{\left\{\abs{X_{n,k}} > a_{n} \right\}} \\
&\leqslant \frac{n \, C(\varepsilon)}{b_{n}} \mathbb{E} \abs{X} I_{\left\{\abs{X} > a_{n} \right\}}
\end{align*}
for some constant $C(\varepsilon)>0$ (non-depending on $n$) and it suffices to prove
\begin{equation}\label{eq:3.9}
\sum_{n=1}^{\infty} \frac{n}{b_{n}} \mathbb{E} \abs{X} I_{\left\{\abs{X} > a_{n} \right\}} < \infty.
\end{equation}
Integrating by parts we get
\begin{equation*}
\mathbb{E} \abs{X} I_{\left\{\abs{X} > a_{n} \right\}} = a_{n} \mathbb{P} \left\{\abs{X} > a_{n} \right\} + \int_{a_{n}}^{\infty} \mathbb{P} \left\{\abs{X} > t \right\} \mathrm{d}t
\end{equation*}
so that \eqref{eq:3.9} turns into
\begin{equation}\label{eq:3.10}
\sum_{n=1}^{\infty} \left(\frac{n a_{n}}{b_{n}} \mathbb{P} \left\{\abs{X} > a_{n} \right\} + \frac{n}{b_{n}} \int_{a_{n}}^{\infty} \mathbb{P} \left\{\abs{X} > t \right\} \mathrm{d}t \right)
\end{equation}
Recalling that
\begin{align*}
\sum_{n=1}^{\infty} \frac{n a_{n}}{b_{n}} \mathbb{P} \left\{\abs{X} > a_{n} \right\} &\leqslant C \sum_{n=1}^{\infty} \log \left(\frac{a_{n} b_{n}}{s_{n}} \right) \frac{n}{\Log \left(n^{\delta} \right)} \mathbb{P} \left\{\abs{X} > a_{n} \right\} \\
& \leqslant C \sum_{n=1}^{\infty} \log \left(\frac{a_{n} b_{n}}{s_{n}} \right) \mathbb{P} \left\{\abs{X} > a_{n} \right\} \int_{0}^{n} \frac{1}{\Log \left(u^{\delta} \right)} \, \mathrm{d}u \\
&= C \int_{0}^{\infty} \frac{1}{\Log \left(u^{\delta} \right)} \sum_{\left\{n \colon n > u \right\}} \Log \left(\frac{a_{n} b_{n}}{s_{n}} \right) \mathbb{P} \left\{\abs{X} > a_{n} \right\} \mathrm{d}u \\
& \leqslant C \int_{0}^{\infty} \frac{1}{\Log \left(u^{\delta} \right)} \int_{\lfloor u \rfloor}^{\infty} \Log \left[\frac{a(t) b(t)}{s(t)} \right] \mathbb{P} \left\{\abs{X} > a(t) \right\} \mathrm{d}t \mathrm{d}u \\
&\leqslant C(\delta) \int_{0}^{\infty} \frac{1}{\Log \, u} \int_{\lfloor u \rfloor}^{\infty} \Log \left[\frac{a(t) b(t)}{s(t)} \right] \mathbb{P} \left\{a^{-1} \left(\abs{X} \right) > t \right\} \mathrm{d}t \mathrm{d}u
\end{align*}
for some $C(\delta) > 0$, and
\begin{align*}
\sum_{n=1}^{\infty} \frac{n}{b_{n}} \int_{a_{n}}^{\infty} \mathbb{P} \left\{\abs{X} > t \right\} \mathrm{d}t &=\int_{0}^{\infty} \mathbb{P} \left\{\abs{X} > t \right\} \sum_{\left\{n\colon a_{n} \leqslant t \right\}} \frac{n}{b_{n}} \mathrm{d}t \\
&\leqslant \int_{0}^{\infty} \mathbb{P} \left\{\abs{X} > t \right\} \int_{0}^{\lfloor a^{-1}(t) \rfloor} \frac{u+1}{b(u)} \mathrm{d}u \, \mathrm{d}t \\
&\leqslant C \int_{0}^{\infty} \mathbb{P} \left\{\abs{X} > t \right\} \int_{0}^{\lfloor a^{-1}(t) \rfloor} \frac{u}{b(u)} \mathrm{d}u \, \mathrm{d}t
\end{align*}
we conclude the convergence of the series \eqref{eq:3.10}, which ensures \eqref{eq:3.8}. Since
\begin{align*}
\mathbb{P} & \left\{\abs{\frac{1}{b_{n}} \sum_{k=1}^{n} \left(X_{n,k} - \mathbb{E} \, X_{n,k} \right)} > \varepsilon \right\} \leqslant \\
& \qquad \leqslant \mathbb{P} \left\{\abs{\frac{1}{b_{n}} \sum_{k=1}^{n} \left(X_{n,k}' - \mathbb{E} \, X_{n,k}' \right)} + \abs{\frac{1}{b_{n}} \sum_{k=1}^{n} \left(X_{n,k}'' - \mathbb{E} \, X_{n,k}'' \right)} > \varepsilon \right\} \\
& \qquad \leqslant \mathbb{P} \left\{\abs{\frac{1}{b_{n}} \sum_{k=1}^{n} (X_{n,k}' - \mathbb{E} \, X_{n,k}')} > \frac{\varepsilon}{2} \right\} + \mathbb{P} \left\{\abs{\frac{1}{b_{n}} \sum_{k=1}^{n} (X_{n,k}'' - \mathbb{E} \, X_{n,k}'')} > \frac{\varepsilon}{2} \right\},
\end{align*}
\eqref{eq:3.7} and \eqref{eq:3.8} yields the thesis.
\end{proofTheorem1}

\begin{proofCorollary1}
Since for all $\varepsilon > 0$,
\begin{align*}
& \sum_{n=1}^{\infty} \mathbb{P} \left\{\abs{\frac{1}{b_{n}} \sum_{k=1}^{n} c_{n,k} \left(X_{n,k} - \mathbb{E} \, X_{n,k} \right)} > \varepsilon \right\} = \\
& \; \; = \sum_{n=1}^{\infty} \mathbb{P} \left\{\abs{\frac{1}{b_{n}} \sum_{k=1}^{n} c_{n,k}^{+} \left(X_{n,k} - \mathbb{E} \, X_{n,k} \right) - \frac{1}{b_{n}} \sum_{k=1}^{n} c_{n,k}^{-} \left(X_{n,k} - \mathbb{E} \, X_{n,k} \right)} > \varepsilon \right\} \\
%& \leqslant \sum_{n=1}^{\infty} \mathbb{P} \left\{\abs{\frac{1}{b_{n}} \sum_{k=1}^{n} c_{n,k}^{+} \left(X_{n,k} - \mathbb{E} \, X_{n,k} \right)} + \abs{\frac{1}{b_{n}} \sum_{k=1}^{n} c_{n,k}^{-} \left(X_{n,k} - \mathbb{E} \, X_{n,k} \right)} > \varepsilon \right\} \\
& \; \; \leqslant \sum_{n=1}^{\infty} \mathbb{P} \left\{\abs{\frac{1}{b_{n}} \sum_{k=1}^{n} c_{n,k}^{+} (X_{n,k} - \mathbb{E} \, X_{n,k})} > \frac{\varepsilon}{2} \right\} + \sum_{n=1}^{\infty} \mathbb{P} \left\{\abs{\frac{1}{b_{n}} \sum_{k=1}^{n} c_{n,k}^{-} (X_{n,k} - \mathbb{E} \, X_{n,k})} > \frac{\varepsilon}{2} \right\}
\end{align*}
where $c_{n,k}^{+} = \max \{c_{n,k}, 0 \} \geqslant 0$ and $c_{n,k}^{-} = \max \{- c_{n,k}, 0 \} \geqslant 0$, we shall assume that the triangular array $\{c_{n,k}, \, 1 \leqslant k \leqslant n, \, n \geqslant 1 \}$ is nonnegative. Thereby, from Lemma 1 of \cite{Lita16b}, $\{c_{n,k} X_{n,k}, \, 1 \leqslant k \leqslant n, \, n \geqslant 1 \}$ is row-wise END with dominating sequence $\{M_{n}, \, n \geqslant 1 \}$.

For $p \geqslant 1$, we have
\begin{equation*}
\sum_{k=1}^{n} \left[\mathbb{E} \left(c_{n,k}^{2} X_{n,k}^{2} I_{\left\{\lvert c_{n,k} X_{n,k} \rvert \leqslant a_{n} \right\}} \right) + a_{n}^{2} \mathbb{P} \left\{\lvert c_{n,k} X_{n,k} \rvert > a_{n} \right\} \right] \leqslant C n \mathbb{E} \, X^{2}
\end{equation*}
from Lemma 2.1 of~\cite{Gut92} where the constant $C$ involves ${\displaystyle \sup_{n \geqslant 1} \, \max_{1 \leqslant k \leqslant n} \abs{c_{n,k}} < \infty}$. Fixing $\delta > 0$ (arbitrarily) and setting $a_{n} = (2 - p) n^{1/p}/(2 \delta p)$, $b_{n} = n^{1/p}$, $s_{n} = C n \mathbb{E} \, X^{2}$, it follows
\begin{gather*}
\frac{s_{n}}{a_{n} b_{n}} = \frac{(2 \delta p) \, C \, \mathbb{E} \, X^{2}}{(2 - p) \, n^{2/p - 1}} \longrightarrow 0 \; \; \textnormal{as} \; \; n \rightarrow \infty, \\
\frac{s_{n}}{a_{n}^{2} \log n} = \frac{(2 \delta p)^{2} \, C \, \mathbb{E} \, X^{2}}{(2 - p)^{2} \, n^{2/p - 1} \log n} \longrightarrow 0 \; \; \textnormal{as} \; \; n \rightarrow \infty, \\
\frac{b_{n} \, \Log \left(a_{n} b_{n}/s_{n} \right)}{a_{n} \, \Log \left(n^{\delta} \right)} = \frac{(2 \delta p) n^{1/p} \log \left[(2 - p) n^{2/p - 1}/(2 \delta p \, C \, \mathbb{E} X^{2}) \right]}{(2 - p) n^{1/p} \, \Log \left(n^{\delta} \right)} \sim 2 \; \; \textnormal{as} \; \; n \rightarrow \infty,
\end{gather*}
and conditions (b), (c) and (d) of Theorem~\ref{thr:1} hold. Since $(2 \delta p t)^{p}/(2 - p)^{p}$ is an asymptotic inverse of $a(t) = (2 - p) t^{1/p}/(2 \delta p)$ (see \cite{Bingham87}, page $28$) we have
\begin{gather}
\int_{0}^{\infty} \frac{1}{\Log \, u} \int_{\lfloor u \rfloor}^{\infty} \Log \left[\frac{a(t) b(t)}{s(t)} \right] \mathbb{P} \left\{a^{-1} \left(\abs{X} \right) > t \right\} \mathrm{d}t \mathrm{d}u \leqslant C(\delta) \int_{0}^{\infty} \mathbb{P} \left\{\abs{X} > y \right\} y^{2p - 1} \mathrm{d}y, \label{eq:3.11} \\
\int_{0}^{\infty} \mathbb{P} \left\{\abs{X} > t \right\} \int_{0}^{\lfloor a^{-1}(t) \rfloor} u^{1 - 1/p} \mathrm{d}u \mathrm{d}t \leqslant C(\delta) \int_{0}^{\infty} \mathbb{P} \left\{\abs{X} > t \right\} t^{2p - 1} \mathrm{d}t \label{eq:3.12}
\end{gather}
According to Lemma 2.4 of \cite{Petrov95} (see page 61), assumptions (e) and (f) are fulfilled which establishes the thesis for $p \geqslant 1$. For $1/2 < p < 1$, we have
\begin{align*}
\sum_{k=1}^{n} & \left[\mathbb{E} \left(c_{n,k}^{2} X_{n,k}^{2} I_{\left\{\lvert c_{n,k} X_{n,k} \rvert \leqslant a_{n} \right\}} \right) + a_{n}^{2} \mathbb{P} \left\{\lvert c_{n,k} X_{n,k} \rvert > a_{n} \right\} \right] \leqslant \\
& \leqslant a_{n}^{2 - 2p} \sum_{k=1}^{n} \mathbb{E} \, \lvert c_{n,k} X_{n,k} \rvert^{2p} I_{\left\{\lvert c_{n,k} X_{n,k} \rvert \leqslant a_{n} \right\}} + a_{n}^{2-2p} \sum_{k=1}^{n} \mathbb{E} \, \lvert c_{n,k} X_{n,k} \rvert^{2p} \\
&\leqslant C n a_{n}^{2 - 2p} \, \mathbb{E} \, \abs{X}^{2p}
\end{align*}
with $C$ depending on ${\displaystyle \sup_{n \geqslant 1} \, \max_{1 \leqslant k \leqslant n} \abs{c_{n,k}} < \infty}$. Taking $s_{n} = C n a_{n}^{2 - 2p} \, \mathbb{E} \, \abs{X}^{2p}$, $a_{n} = n^{1/p}/(2 \delta)$ and $b_{n} = n^{1/p}$ we still obtain
\begin{gather*}
\frac{s_{n}}{a_{n} b_{n}} = \frac{C \, (2 \delta)^{2p - 1} \, \mathbb{E} \abs{X}^{2p}}{n} \longrightarrow 0 \; \; \textnormal{as} \; \; n \rightarrow \infty, \\
\frac{s_{n}}{a_{n}^{2} \log n} = \frac{C \, (2 \delta)^{2p} \, \mathbb{E} \abs{X}^{2p}}{n \, \log n} \longrightarrow 0 \; \; \textnormal{as} \; \; n \rightarrow \infty, \\
\frac{b_{n} \, \Log \left(a_{n} b_{n}/s_{n} \right)}{a_{n} \, \Log \left(n^{\delta} \right)} = \frac{(2 \delta) n^{1/p} \log \left[n/\left((2 \delta)^{2p-1} \, C \, \mathbb{E} \abs{X}^{2p} \right) \right]}{n^{1/p} \, \Log \left(n^{\delta} \right)} \sim 2 \; \; \textnormal{as} \; \; n \rightarrow \infty.
\end{gather*}
Again, conditions (b), (c) and (d) of Theorem~\ref{thr:1} are satisfied, as well as \eqref{eq:3.11} and \eqref{eq:3.12}, yielding the conclusion for $1/2 < p < 1$. Finally, supposing $0 < p \leqslant 1/2$, Lemma~\ref{lem:2} and Corollary 2.2 of \cite{Gut92} guarantee
\begin{align*}
& \sum_{n=1}^{\infty} \mathbb{P} \left\{\abs{\sum_{k=1}^{n} c_{n,k} X_{n,k}} > \varepsilon n^{1/p} \right\} \leqslant \\
&\quad \leqslant \sum_{n=1}^{\infty} \sum_{k=1}^{n} \mathbb{P} \left\{\abs{c_{n,k} X_{n,k}} > \frac{\varepsilon n^{1/p}}{\lambda} \right\} + 2 e^{\lambda} \sum_{n=1}^{\infty} M_{n} \left(1 + \frac{\varepsilon^{2p} n^{2}}{\lambda^{2p-1} \sum_{k=1}^{n} \mathbb{E} \abs{c_{n,k} X_{n,k}}^{2p}} \right)^{-\lambda} \\
&\quad \leqslant \sum_{n=1}^{\infty} n \mathbb{P} \left\{\abs{X} > \frac{(\varepsilon/L) n^{1/p}}{\lambda} \right\} + 2 e^{\lambda} \sum_{n=1}^{\infty} M_{n} \left(1 + \frac{(\varepsilon/L)^{2p} n}{\lambda^{2p-1} C} \right)^{-\lambda} < \infty
\end{align*}
provided that, for some $\alpha > 0$, $M_{n} = O \left(n^{\alpha} \right)$, $n \rightarrow \infty$, $\lambda > \alpha + 1$ and a (fixed) constant $L$ such that ${\displaystyle L \geqslant \sup_{n \geqslant 1} \, \max_{1 \leqslant k \leqslant n} \abs{c_{n,k}}}$. The proof is complete.
\end{proofCorollary1}

\begin{proofTheorem2}
Without loss of generality and similarly to the proof of Corollary~\ref{cor:1}, we shall admit that the triangular array $\{c_{n,k}, \, 1 \leqslant k \leqslant n, \, n \geqslant 1 \}$ is nonnegative; otherwise, one can always perform
\begin{equation*}
\sum_{k=1}^{n} c_{n,k} X_{k} = \sum_{k=1}^{n} c_{n,k}^{+} X_{k} - \sum_{k=1}^{n} c_{n,k}^{-} X_{k},
\end{equation*}
where $c_{n,k}^{+} = \max \{c_{n,k}, 0 \} \geqslant 0$ and $c_{n,k}^{-} = \max \{- c_{n,k}, 0 \} \geqslant 0$. Consider $a_{n} = n^{1/p}/\Log^{1/p} \, n$, $b_{n} = n^{1/p} \, \Log^{1 - 1/p} \, n$,
\begin{equation*}
X_{n}' = X_{n} I_{\left\{\abs{X_{n}} \leqslant a_{n} \right\}} - \mathbb{E} \, X_{n} I_{\left\{\abs{X_{n}} \leqslant a_{n} \right\}} + a_{n} I_{\left\{X_{n} > a_{n} \right\}} - a_{n} I_{\left\{X_{n} < -a_{n} \right\}}
\end{equation*}
and
\begin{equation*}
X_{n}'' = X_{n} I_{\left\{\abs{X_{n}} > a_{n} \right\}} - \mathbb{E} \, X_{n} I_{\left\{\abs{X_{n}} > a_{n} \right\}} + a_{n} I_{\left\{X_{n} < -a_{n} \right\}} - a_{n} I_{\left\{X_{n} > a_{n} \right\}}.
\end{equation*}
Therefore, $X_{n}' + X_{n}'' = X_{n} - \mathbb{E} \, X_{n}$. The random variables
\begin{equation*}
X_{n} I_{\left\{\abs{X_{n}} \leqslant a_{n} \right\}} + a_{n} I_{\left\{X_{n} > a_{n} \right\}} - a_{n} I_{\left\{X_{n} < -a_{n} \right\}}
\end{equation*}
are widely orthant dependent with dominating sequence $\{M_{n}, \, n \geqslant 1 \}$ by Lemma 2.1 of \cite{Shen16} since the function $T_{\ell}(t) = \max(\min(t,\ell),-\ell)$ is nondecreasing. Hence, the sequences $\{X_{n}', \, n \geqslant 1\}$ and $\{c_{n,k} X_{k}', \, 1 \leqslant k \leqslant n \}$, for every $n \geqslant 1$, are also widely orthant dependent with dominating sequence $\{M_{n}, \, n \geqslant 1 \}$, as they are nondecreasing transformations of widely orthant dependent random variables with the referred dominating sequence. This means that $\{c_{n,k} X_{k}', \, 1 \leqslant k \leqslant n, \, n \geqslant 1 \}$ is row-wise END with dominating sequence $\{M_{n}, \, n \geqslant 1 \}$.
Putting $X_{n,k} := c_{n,k} X_{k}'$ we have
\begin{equation*}
\sum_{k=1}^{n} \mathbb{E} \, X_{n,k}^{2} \leqslant 4 \sup_{m \geqslant 1} \, \max_{1 \leqslant j \leqslant m} c_{m,j}^{2} \sum_{k=1}^{n} \mathbb{E} \left[X_{k}^{2} I_{\left\{\abs{X_{k}} \leqslant a_{k} \right\}} + a_{k}^{2} I_{\left\{\abs{X_{k}} > a_{k} \right\}} \right] =: s_{n}
\end{equation*}
with ${\displaystyle \sup_{m \geqslant 1} \, \max_{1 \leqslant j \leqslant m} c_{m,j}^{2} < \infty}$ according to \eqref{eq:2.3}. Therefore,
\begin{equation*}
\frac{s_{n}}{a_{n} b_{n}} = \frac{C}{n^{2/p} \, \Log^{1 - 2/p} \, n} \sum_{k=1}^{n} \mathbb{E} \left[X_{k}^{2} I_{\left\{\abs{X_{k}} \leqslant a_{k} \right\}} + a_{k}^{2} I_{\left\{\abs{X_{k}} > a_{k} \right\}} \right] \longrightarrow 0
\end{equation*}
and
\begin{equation*}
\frac{s_{n}}{a_{n}^{2} \, \Log \, n} = \frac{C}{n^{2/p} \, \Log^{1 - 2/p} \, n} \sum_{k=1}^{n} \mathbb{E} \left[X_{k}^{2} I_{\left\{\abs{X_{k}} \leqslant a_{k} \right\}} + a_{k}^{2} I_{\left\{\abs{X_{k}} > a_{k} \right\}} \right] \longrightarrow 0
\end{equation*}
as $n \rightarrow \infty$. Moreover, for each $\delta > 0$,
\begin{equation*}
\frac{b_{n} \, \Log \left(a_{n} b_{n}/s_{n} \right)}{a_{n} \, \Log \left(n^{\delta} \right)} = - \frac{1}{\delta} \, \Log \left(\frac{C}{n^{2/p} \, \Log^{1 - 2/p} \, n} \sum_{k=1}^{n} \mathbb{E} \left[X_{k}^{2} I_{\left\{\abs{X_{k}} \leqslant a_{k} \right\}} + a_{k}^{2} I_{\left\{\abs{X_{k}} > a_{k} \right\}} \right] \right) \longrightarrow \infty
\end{equation*}
as $n \rightarrow \infty$ via Lemma 4 of \cite{Lita15} (see page $62$) and Kronecker's lemma. Thus, from Lemma~\ref{lem:3} we get
\begin{equation}\label{eq:3.13}
\frac{1}{n^{1/p} \, \Log^{1 - 1/p} \, n} \sum_{k=1}^{n} c_{n,k} X_{k}' \overset{\textnormal{a.s.}}{\longrightarrow} 0.
\end{equation}
It suffices to prove
\begin{equation*}
\frac{1}{n^{1/p} \, \Log^{1 - 1/p} \, n} \sum_{k=1}^{n} c_{n,k} X_{k}'' \overset{\textnormal{a.s.}}{\longrightarrow} 0.
\end{equation*}
We have
\begin{align*}
\max_{2^{m} \leqslant n < 2^{m+1}} \abs{\sum_{k=1}^{n} \frac{c_{n,k}}{n^{1/p} \, \Log^{1 - 1/p} n} X_{k}''} &\leqslant C \max_{2^{m} \leqslant n < 2^{m+1}} \frac{1}{n^{1/p} \Log^{1-1/p} n} \sum_{k=1}^{n} \abs{X_{k}''} \\
& \leqslant \frac{C}{\left(2^{m+1} \right)^{1/p} \left(\Log \, 2^{m+1} \right)^{1-1/p}} \sum_{k=1}^{2^{m+1}} \abs{X_{k}''}
\end{align*}
and for any $\varepsilon>0$ we obtain from Lemma 4 of \cite{Lita15},
\begin{align*}
\sum_{m=1}^{\infty} \mathbb{P} & \left\{\frac{1}{\left(2^{m} \right)^{1/p} \left(\Log \, 2^{m} \right)^{1-1/p}} \sum_{k=1}^{2^{m}} \abs{X_{k}''} > \varepsilon \right \} \leqslant \\
& \leqslant \frac{1}{\varepsilon} \sum_{m=1}^{\infty} \frac{1}{\left(2^{m} \right)^{1/p} \left(\Log \, 2^{m} \right)^{1-1/p}} \sum_{k=1}^{2^{m}} \mathbb{E} \abs{X_{k}''} \\
& = \frac{2}{\varepsilon} \sum_{k=1}^{\infty} \mathbb{E} \left[\abs{X_{k}} I_{\left\{\abs{X_{k}} > a_{k} \right\}} + a_{k} I_{\left\{\abs{X_{k}} > a_{k} \right\}} \right] \sum_{\left\{m\colon 2^{m} \geqslant k \right\}} \frac{1}{\left(2^{m} \right)^{1/p} \left(\Log \, 2^{m} \right)^{1-1/p}} \\
& \leqslant \frac{2}{\varepsilon} \sum_{k=1}^{\infty} \mathbb{E} \left[\abs{X_{k}} I_{\left\{\abs{X_{k}} > a_{k} \right\}} + a_{k} I_{\left\{\abs{X_{k}} > a_{k} \right\}} \right] \frac{1}{\Log^{1-1/p} \, k} \sum_{\left\{m\colon 2^{m} \geqslant k \right\}} \frac{1}{\left(2^{m} \right)^{1/p}} \\
& \leqslant C(\varepsilon) \sum_{k=1}^{\infty} \frac{1}{k^{1/p} \, \Log^{1-1/p} k} \mathbb{E} \left[\abs{X_{k}} I_{\left\{\abs{X_{k}} > a_{k} \right\}} + a_{k} I_{\left\{\abs{X_{k}} > a_{k} \right\}} \right] < \infty
\end{align*}
where the summation $\underset{\left\{m\colon 2^{m} \geqslant k \right\}}{\sum}$ is taken over all $m$ such that $2^{m} \geqslant k$. Thus, Borel-Cantelli lemma permits us to conclude
\begin{equation*}
\max_{2^{m-1} \leqslant n < 2^{m}} \abs{\sum_{k=1}^{n} \frac{c_{n,k}}{n^{1/p} \, \Log^{1 - 1/p} n}  X_{k}''} \overset{\textnormal{a.s.}}{\longrightarrow} 0
\end{equation*}
as $m \rightarrow \infty$ , so that
\begin{equation*}
\frac{1}{n^{1/p} \, \Log^{1 - 1/p} n} \sum_{k=1}^{n} c_{n,k} X_{k}'' \overset{\textnormal{a.s.}}{\longrightarrow} 0
\end{equation*}
and the thesis is established.
\end{proofTheorem2}

Looking in detail to the proof of Theorem~\ref{thr:2}, we can infere that Lemma~\ref{lem:3} is sharper than Lemma 3 of \cite{Lita16a} in some scenarios. In fact, we saw that the triangular array $\{c_{n,k} X_{k}', \, 1 \leqslant k \leqslant n, \, n \geqslant 1 \}$ of zero-mean row-wise END random variables with dominating sequence $M_{n} = M$ (for instance) and $a_{n} = n^{1/p}/\Log^{1/p} \, n$, $b_{n} = n^{1/p} \, \Log^{1 - 1/p} \, n$, $s_{n} := C \sum_{k=1}^{n} \mathbb{E} \big[X_{k}^{2} I_{\left\{\abs{X_{k}} \leqslant a_{k} \right\}} + a_{k}^{2} I_{\left\{\abs{X_{k}} > a_{k} \right\}} \big]$ verify all assumptions of Lemma~\ref{lem:3} leading to \eqref{eq:3.13}. However, condition (iii) of Lemma 3 in \cite{Lita16a} is not satisfied, i.e. ${\displaystyle \limsup_{n \rightarrow \infty}} \, a_{n} \sqrt{\Log \, n/s_{n}} = \infty$. Even using Bernstein's inequality (Lemma 2 of \cite{Lita16a}) instead of Bennet's inequality in the conception of prior Lemma 3, we would have
\begin{equation*}
\mathbb{P} \left\{\frac{1}{b_{n}} \sum_{k=1}^{n} X_{n,k} > \varepsilon \right\} \leqslant M_{n} \exp \left\{-\frac{\varepsilon^{2} b_{n}^{2}}{2(\varepsilon a_{n} b_{n} + s_{n})} \right\}.
\end{equation*}
For the sequences $a_{n}$, $b_{n}$ and $s_{n}$ earlier chosen, it would follow $\varepsilon^{2} b_{n}^{2}/(\varepsilon a_{n} b_{n} + s_{n}) \sim \varepsilon \, \Log \, n$ as $n \rightarrow \infty$ and convergence \eqref{eq:3.13} would not be guaranteed.

Naturally, sharper rates (i.e. norming constants) on strong laws of large numbers can be achieved as long as sharper exponencial probability inequalities can be founded.

\section*{Acknowledgements}

This work is a contribution to the Project UID/GEO/04035/2013, funded by FCT - Funda\c{c}\~{a}o para a Ci\^{e}ncia e a Tecnologia, Portugal.



\end{document}